\documentclass[12pt]{article}
\usepackage{amsfonts, amsmath}
\usepackage{amssymb}
\usepackage{amsmath,amscd}
\usepackage{lipsum}
\usepackage[pdftex]{graphicx}
\usepackage[hang]{subfigure}

\usepackage[margin=25pt,font=small,labelfont=bf]{caption}

\newcommand{\bna}{\begin{eqnarray}}
\newcommand{\ena}{\end{eqnarray}}
\newcommand{\ba}{\begin{eqnarray*}}
\newcommand{\ea}{\end{eqnarray*}}

\newcommand{\bs}[1]{}

\newtheorem{theorem}{Theorem}
\newtheorem{corollary}{Corollary}

\def\p{{\bf p}}

\def\q{{\bf q}}

\def\v{{\bf v}}

\begin{document}

\title{Straight Line motion with rigid sets}

\author{Robert Connelly and 
Luis Montejano }
\maketitle

\begin{abstract}  If one is given a rigid triangle in the plane or space, we show that the only motion possible, where each vertex of the triangle moves along a straight line, is given by a hypocycloid line drawer in the plane, and a natural extension in three-space.  Each point lies on a circle which rolls around, without slipping,  inside a larger circle of twice its diameter.
\end{abstract}
\section{Introduction}
Consider three straight lines in the plane or in three-space.  When can you continuously move a point on each line such that all the pairwise distances between them stay constant?  One case is when the three lines are parallel. Figure \ref{fig:hypocycloid} shows another way.  Are these the only possibilities?  We answer the question affirmatively with Theorem \ref{thm:plane-main} in the plane, and Theorem \ref{thm:space-main} in higher dimensions.

\begin{figure}[here]
    \begin{center}
        \includegraphics[width=0.78\textwidth]{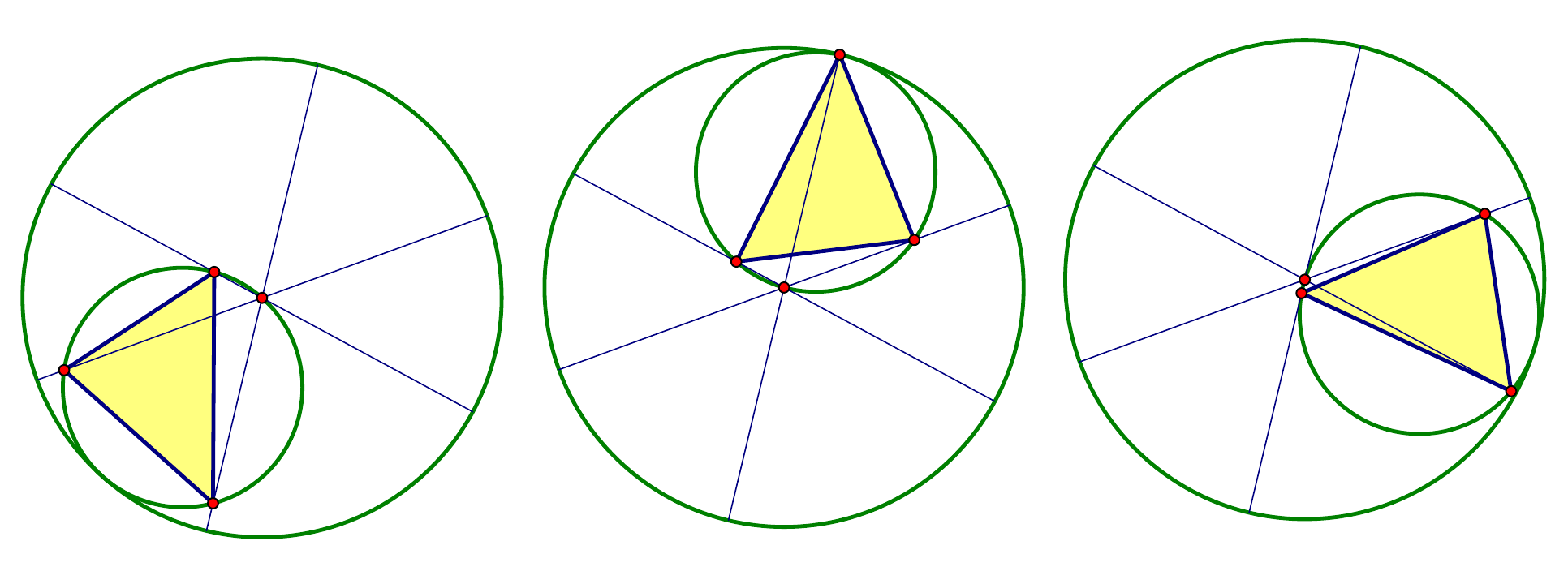}%
        \end{center}
    \caption{}
    \label{fig:hypocycloid}
    \end{figure}
    
The example of Figure  \ref{fig:hypocycloid} is known as a hypocycloid straight-line mechanism, and Figure \ref{fig:S16-2} shows a model from the Cornell Reuleaux collection. \cite{Reuleaux}.

\begin{figure}[here]
    \begin{center}
        \includegraphics[width=0.30\textwidth]{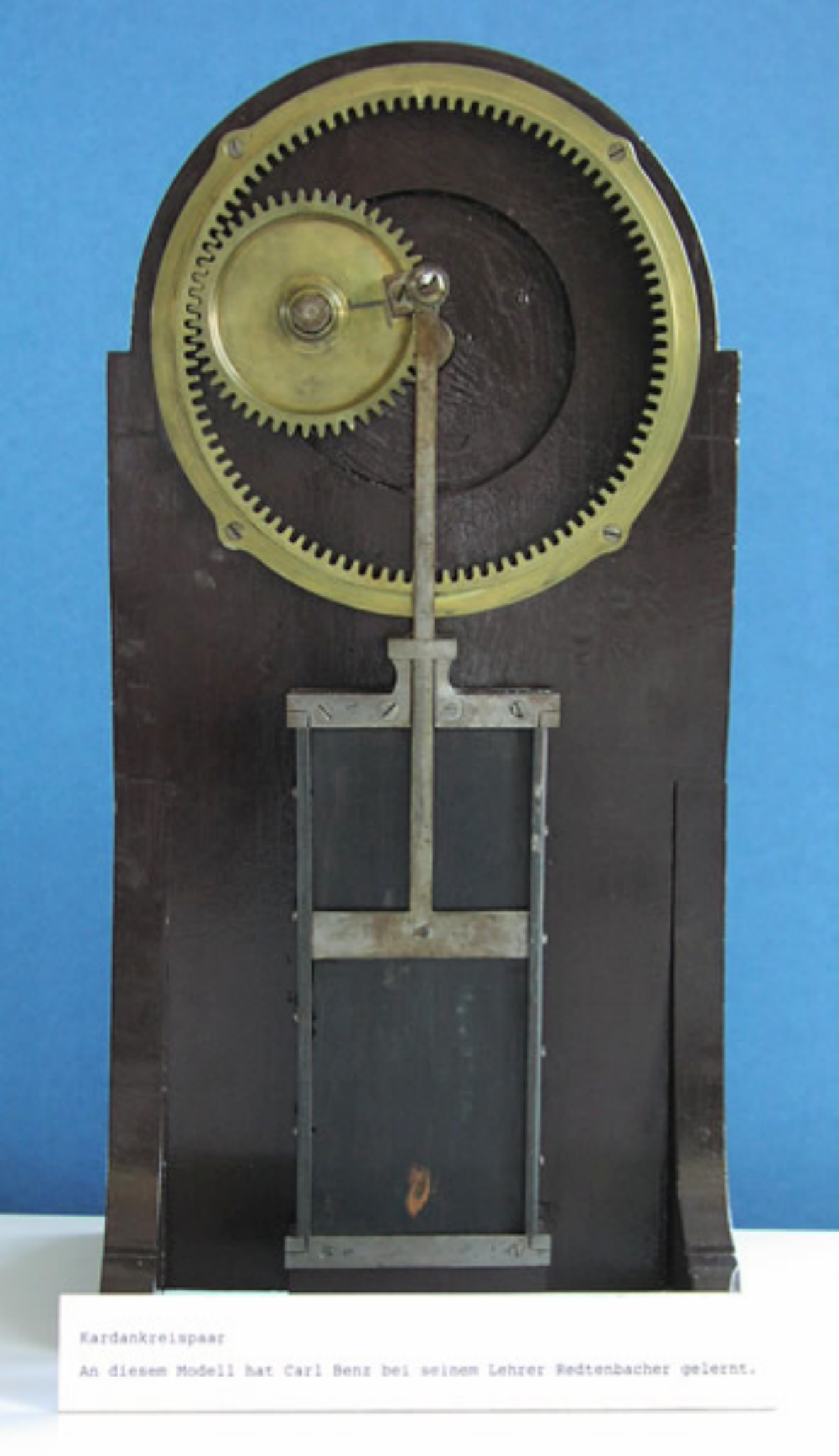}%
        \end{center}
    \caption{}
    \label{fig:S16-2}
    \end{figure}

It is easy to see from Figure \ref{fig:hypocycloid} that as the inner circle rolls  around the central point with a fixed radius, the triangle, formed from the other points of intersection with the three fixed lines form a triangle with fixed internal angles, and therefore the triangle moves rigidly.

In dimension three, one can form a cylinder rotating inside a larger cylinder of twice the diameter.  But in the plane and three-space, the three lines all have a fourth line that is perpendicular and incident to all three lines.  

\section{The planar case}

We first describe the motion of a segment of fixed length $d_{12}$ sliding between two fixed non-parallel lines, $L_1$ and $L_2$, first in the plane.  We assume, without loss of generality, that $L_1$ and $L_2$ intersect at the origin and are determined by two unit vectors $\v_1, \v_2$, where $\v_1 \cdot \v_2 =\cos \alpha_{12} =c_{12}$,  $\alpha_{12}$ being the angle between $\v_1$ and  $\v_2$, and $-1<c_{12} <1$.  Let $t_1, t_2$ be the oriented distance from the origin of the points $\p_1, \p_2$, so that $\p_1=t_1\v_1$, and $\p_2=t_2\v_2$.  Then treating the square of vector as the dot product with itself,
\begin{equation}\label{eqn:length}
d_{12}^2=(\p_1-\p_2)^2=(t_1\v_1-t_2\v_2)^2=t_1^2+t_2^2-2c_{12}t_1t_2.
\end{equation}
Thus in $t_1, t_2$ space the configurations of the line segment form an ellipse centered at the origin whose major and minor axes are at $45^{\circ}$ from the $t_1, t_2$ axes.  Thus there are constants $a_{12}=d_{12}/\sqrt{2(1-c_{12})},  b_{12}=d_{12}/\sqrt{2(1+c_{12})}$ such that 
\begin{equation}\label{eqn:trig}
t_1=a_{12}\cos \theta - b_{12}\sin \theta, \,\,\, t_2=a_{12}\cos \theta + b_{12}\sin \theta
\end{equation}
describes the full range of motion of the line segment for $0 \le \theta \le 2\pi$.  It is also clear from Figure \ref{fig:sliding-edge} that the length of the image of each $\p_i$ on $L_i$, $i=1,2$ is an interval of length $2d_{12}/\sin \alpha_{12}$ centered about the intersection of $L_1$ and $L_2$.   
\begin{figure}[here]
    \begin{center}
        \includegraphics[width=0.4\textwidth]{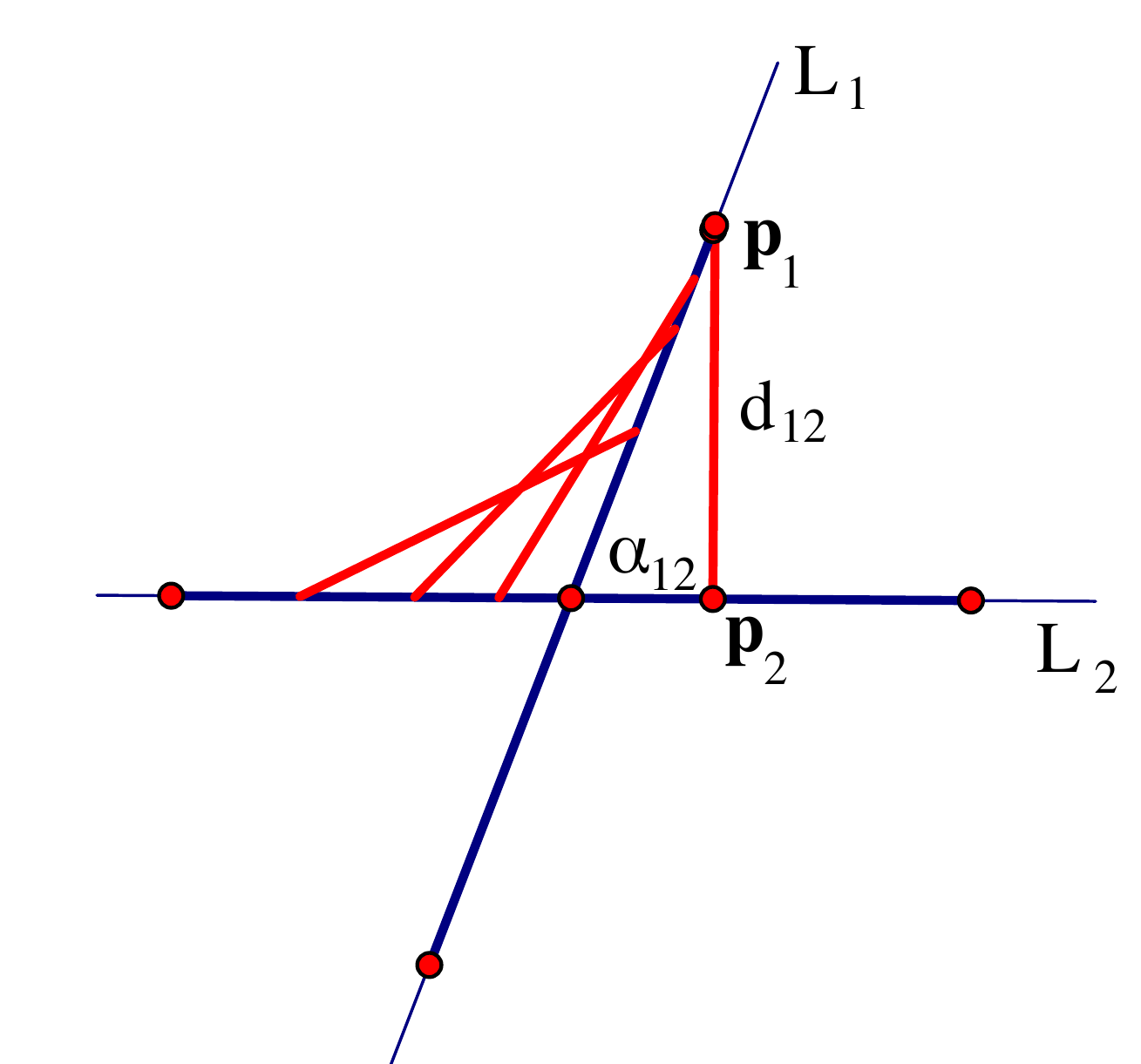}%
        \end{center}
    \caption{}
    \label{fig:sliding-edge}
    \end{figure}
We can now state the situation for the plane.

\begin{theorem}\label{thm:plane-main}If a triangle, with fixed edge lengths, continuously moves with each vertex $\p_i$ on a line $L_i$, $i=1,2,3$ in the plane, then either all the lines are parallel or they all intersect at a point forming a hypocycloid straight line drawer as above, with the range of each point an interval of the same length on each line, while the midpoint of each range is the intersection of the two lines.
\end{theorem}

\begin{proof}Choose any two of the the three lines, say $L_1$ and $L_2$.  The parametrization discussed above shows that any position of the $d_{12}$ segment can described by the equations in (\ref{eqn:trig}).  These define the positions of $\p_1$ and $\p_2$ as a function of $\theta$, for all $0 \le \theta \le 2\pi$.   Then the position of $\p_3$ is determined also as a function of $\theta$, $\p_3(\theta)$, since it is carried along rigidly.  The image of $\p_3$ is a continuous curve in the plane and it is symmetric about the intersection of $L_1$ and $L_2$. (Actually it is generally an ellipse, also, but that is not needed for this part of the argument.) If that image is to be in a straight line $L_3$, then $L_3$ should intersect the intersection of $L_1$ and $L_2$.  So if $\p_3(\theta)$ satisfies the equations corresponding to (\ref{eqn:length}) for $13$ and $23$ replacing $12$, for an interval of values of $\theta$, then it must satisfy those equations for all $\theta$ and the image of each $\p_i$ in $L_i$ is the same length for $i=1,2$.  But applying this argument to another pair such as $13$, shows that all the images are the same length. $\Box$
\end{proof}
\section{The higher dimensional case}

In higher dimensions, each pair of non-parallel lines $L_i, L_j, \,\,i, j = 1,2,3$ have two points $\q_{ij} \in L_i$, for $i \ne j$, such that $\q_{ij}-\q_{ji}$ is perpendicular to both $L_i$ and $L_j$.  The shortest distance between  $L_i$ and $L_j$ is $|\q_{ij}-\q_{ji}|=D_{ij}=D_{ji}$. We construct six variables $t_{ij},\,\, i,j =1,2,3,\,\, i\ne j$, where $\p_i = \q_{ij} + t_{ij}\v_i$. Then we can project orthogonally into a plane spanned by the vectors $\v_i, \v_j$ to get the analog  of Equation (\ref{eqn:length}), which is the following representing the three equations for each of the three edge lengths $d_{ij}=d_{ji}$ of the triangle:
\begin{equation}\label{eqn:high-d-length}
d_{ij}^2-D_{ij}^2=(\p_i-\p_j)^2=(t_{ij}\v_i-t_{ji}\v_j)^2=t_{ij}^2+t_{ji}^2-2c_{ij}t_{ij}t_{ji}.
\end{equation}
Similar to Equation (\ref{eqn:trig}), we define constants $a_{ij}=\sqrt{d_{ij}^2-D_{ij}^2}/\sqrt{2(1-c_{ij})}, \\ b_{ij}=\sqrt{d_{ij}^2-D_{ij}^2}/\sqrt{2(1+c_{ij})}$ such that $a_{ij}=a_{ji}, \, b_{ij}=b_{ji}$ and 
\begin{equation}\label{eqn:trig-2}
t_{ij}=a_{ij}\cos \theta_{ij} - b_{ij}\sin \theta_{ij}, \,\,\, t_{ji}=a_{ij}\cos \theta_{ij} + b_{ij}\sin \theta_{ij}
\end{equation}
and $\theta_{ij}=\theta_{ji}$ is the parameter as before. This gives three separate parameterizations for each line segment between each pair of lines.  Furthermore the equations of (\ref{eqn:high-d-length}) give a complete description of position of each $\p_i$, two ways for each line, in terms of $t_{ij}$ and $t_{ik}$, where $|t_{ij}-t_{ik}|=|\q_{ij}-\q_{ik}|$, a constant.  So $t_{ij}=t_{ik}  + e_{ijk}$, where $e_{ijk}= \pm |\q_{ij}-\q_{ik}|$ is a constant.  See Figure \ref{fig:coordinates} for a view of these coordinates.  

\newpage
\begin{figure}[here]
    \begin{center}
        \includegraphics[width=0.7\textwidth]{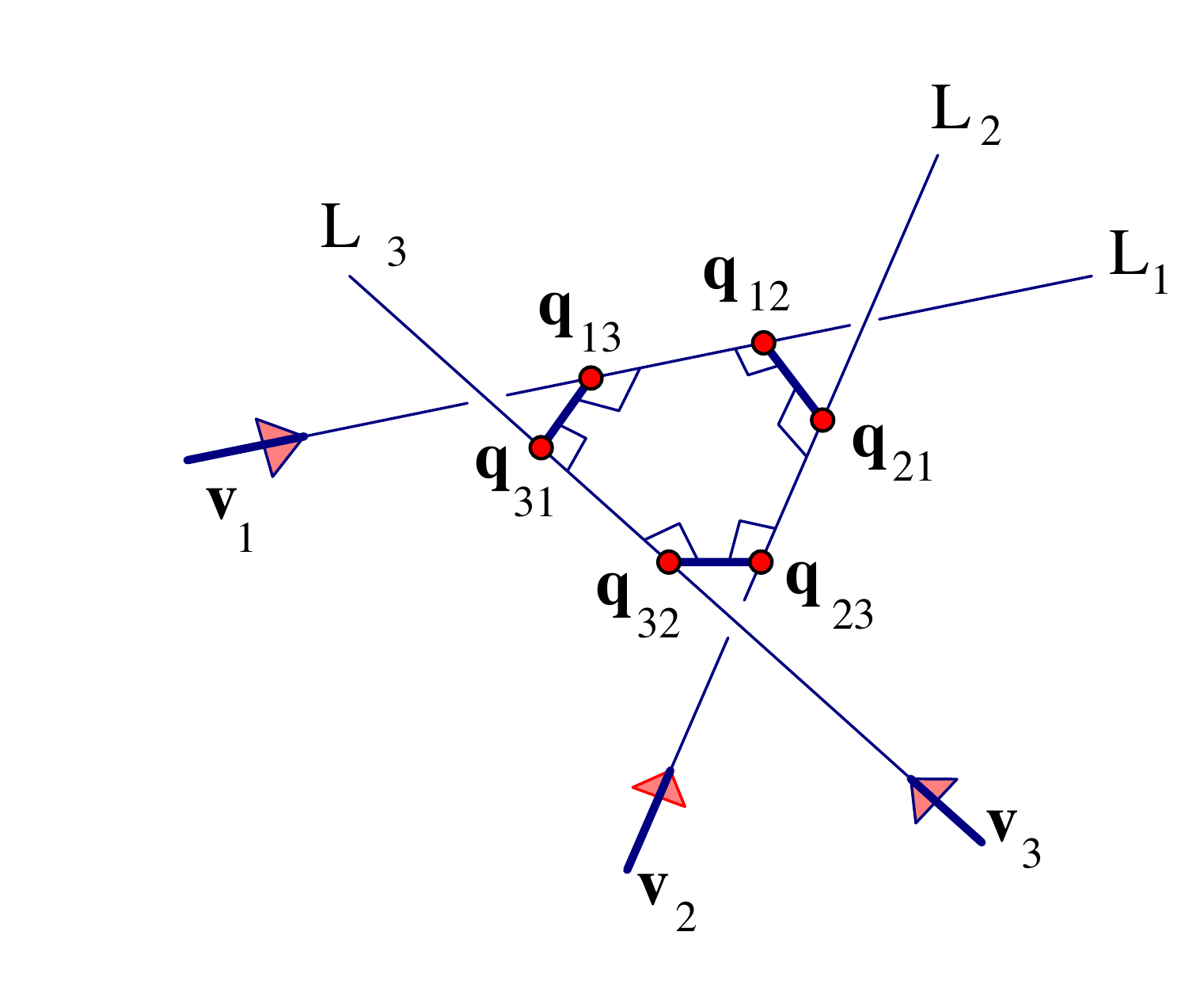}%
        \end{center}
    \caption{}
    \label{fig:coordinates}
    \end{figure}

\begin{theorem}\label{thm:space-main}If a triangle, with fixed edge lengths, continuously moves with each vertex $\p_i$ on a line $L_i$, $i=1,2,3$ in a Euclidean space, then either all the lines are parallel or they all intersect a line perpendicular all of them, with the range of each point an interval of the same length on each line, while the midpoint of each range is the intersection of the two lines.
\end{theorem}

\begin{proof}Start with Equation (\ref{eqn:high-d-length}) for the $L_1$, $L_2$ lines defining the variable $\theta_{12}=\theta_{21}$ which, in turn defines the positions for $\p_1(\theta_{12})$ on $L_1$ and  $\p_2(\theta_{12})$ on $L_2$.  Note that  $t_{12} =t_{13}+e_{123}$ and $t_{21} =t_{23}+ e_{213}$.  So we can regard $t_{13}$ and  $t_{23}$ as linear functions of $t_{12}$ and $t_{21}$.
Similarly, $t_{32} = t_{31} + e_{321}$, so we can regard $t_{32}$, as a function of $t_{31}$.  So if we subtract the $13$ and $23$ equations for (\ref{eqn:high-d-length}),  we are only left with linear terms in $t_{31}$.  The squared terms have cancelled.  Thus either $t_{31}$ has no term in the difference of the two equations, or $t_{31}$ is a non-zero rational function of  $t_{12}$ and $t_{21}$ and thus $\theta_{12}$.

If the linear $t_{31}$ terms have disappeared, then the $12$ and $13$ equations imply the $23$ equation, which means that the $12$ edge of the triangle makes a full $360^{\circ}$ turn and the image of $\p_1$ in $L_1$ is symmetric about $\q_{12}$.  If the $t_{31}$ term is a non-zero rational function of  $t_{12}$ and $t_{21}$ and thus $\theta_{12}$, and again the image of $\p_1$ in $L_1$ is symmetric about $\q_{12}$.

Applying the above argument to each equation for (\ref{eqn:high-d-length}) for $ij$, we see that each image of $\p_i$ in $L_i$ is symmetric about $\q_i$.  If all the lines lie in a three-dimensional Euclidean space, and no two lines $L_i$ are parallel, then the midpoints of the images of each $\p_i$ must lie on a line perpendicular to all of the $L_i$, and as claimed.  So these lines are just a three-dimensional ``lift" of the two-dimensional case.

If the lines span a higher dimensional space, then there is a non-zero vector perpendicular to each $\v_i$, for $i=1,2,3$, and the lines $L_i$ can be projected orthogonally into a three-dimensional space, and the projection of the points on the projected lines will be a mechanism, one dimension lower. Then similarly the three-dimensional mechanism comes from a two-dimensional mechanism, from the argument above. $\Box$
\end{proof}

The following are some immediate corollaries.

\begin{corollary}If a polygon in the plane is continuously moves as a rigid body so that each vertex stays on a straight line, then those vertices all lie on a circle or the three lines are parallel.
\end{corollary}

\begin{corollary}Given three lines $L_1, L_2, L_3$ on a Euclidean space such that they are are not all perpendicular to a fourth line, then there are at most $8$ configurations of a triangle $\p_1,\p_2,\p_3$ with fixed edge lengths, where $\p_i $ is on $L_i$, $i=1,2,3$.
\end{corollary}

This is because each edge length is determined by a degree two equation, and B\'ezout's theorem \cite{Bezout} implies that there are at most $2^3=8$ individual solutions.

\bibliographystyle{plain}
\bibliography{NSF-10,framework}

\end{document}